\documentclass[11pt]{amsart}

\usepackage{amsmath,amssymb,latexsym,soul,cite,mathrsfs,amsfonts}
\usepackage{subfigure}
\usepackage{float}

\usepackage{color,enumitem,graphicx}
\usepackage[colorlinks=true,urlcolor=blue,
citecolor=red,linkcolor=blue,linktocpage,pdfpagelabels,
bookmarksnumbered,bookmarksopen]{hyperref}
\usepackage[english]{babel}
\usepackage{comment}

\usepackage[left=2.9cm,right=2.9cm,top=2.8cm,bottom=2.8cm]{geometry}
\usepackage[hyperpageref]{backref}

\usepackage[colorinlistoftodos]{todonotes}
\makeatletter
\providecommand\@dotsep{5}
\def\listtodoname{List of Todos}
\def\listoftodos{\@starttoc{tdo}\listtodoname}
\makeatother

\usepackage{verbatim} 

\numberwithin{equation}{section}

\newtheorem{theorem}{Theorem}[section]
\newtheorem{prop}[theorem]{Proposition}
\newtheorem{lem}[theorem]{Lemma}
\newtheorem{cor}[theorem]{Corollary}
\newtheorem{rem}{Remark}
\newtheorem{definition}{Definition}
\newcommand\restr[2]{{
  \left.\kern-\nulldelimiterspace 
  #1 
  \vphantom{\big|} 
  \right|_{#2} 
  }}

\title[Extremal parameters of a Kirchhoff type equation]
{The bifurcation diagram of an elliptic Kirchhoff-type equation with respect to the stiffness of the material 
}

\author[K. Silva]{Kaye Silva}

\address[K. Silva]{\newline\indent
	Instituto de Matem\'atica e Estat\'istica.   
	\newline\indent 
	Universidade Federal de Goi\'as,
	\newline\indent
	74001-970, Goi\^ania, GO, Brazil}
\email{\href{mailto:kayeoliveira@hotmail.com}{kayeoliveira@hotmail.com}, \href{mailto:kaye_0liveira@ufg.br}{kaye\_0liveira@ufg.br}}


\begin{document}

\begin{abstract}
	We study a superlinear and subcritical Kirchhoff type equation which is variational and depends upon a real parameter $\lambda$. The nonlocal term forces some of the fiber maps associated with the energy functional to have two critical points. This suggest multiplicity of solutions and indeed we show the existence of a local minimum and a mountain pass type solution. We characterize the first parameter $\lambda_0^*$ for which the local minimum has non-negative energy when $\lambda\ge \lambda_0^*$. Moreover we characterize the extremal parameter $\lambda^*$ for which if $\lambda>\lambda^*$, then the only solution to the Kirchhoff equation is the zero function. In fact, $\lambda^*$ can be characterized in terms of the best constant of Sobolev embeddings. We also study the asymptotic behavior of the solutions when $\lambda\downarrow 0$.
\end{abstract}

\maketitle

\section{Introduction}
In this work we study the following Kirchhoff type equation
\begin{equation}\label{p}
\left\{
\begin{aligned}
-\left(a+\lambda\int |\nabla u|^2\right)\Delta u&= |u|^{\gamma-2}u &&\mbox{in}\ \ \Omega, \\
u&=0                                   &&\mbox{on}\ \ \partial\Omega,
\end{aligned}
\right.
\end{equation}
where $a>0$, $\lambda>0$ is a parameter, $\Delta$ is the Laplacian operator and $\Omega\subset \mathbb{R}^3$ is a bounded regular domain. Let $H_0^1(\Omega)$ denote the standard Sobolev space and $\Phi_\lambda:H_0^1(\Omega)\to \mathbb{R}$ the energy functional associated with (\ref{p}), that is
\begin{equation}\label{energyfunctional}
\Phi_\lambda(u)=\frac{a}{2}\int |\nabla u|^2+\frac{\lambda}{4}\left(\int |\nabla u|^2\right)^2-\frac{1}{\gamma}\int |u|^\gamma.
\end{equation} 
We observe that $\Phi_\lambda$ is a $C^1$ functional. By definition a solution to equation (\ref{p}) is a critical point of $\Phi_\lambda$. Our main result is the following
\begin{theorem}\label{T} Suppose $\gamma\in(2,4)$. Then there exist parameters $0<\lambda_0^*<\lambda^*$ and $\varepsilon>0$ such that:
	\begin{enumerate}
		\item[1)] For each $\lambda\in(0,\lambda^*]$ problem \eqref{p} has a positive solution $u_\lambda$ which is a global minimizer for $\Phi_\lambda$ when $\lambda\in(0,\lambda_0^*]$, while $u_\lambda$ is a local minimizer for $\Phi_\lambda$ when $\lambda\in(\lambda_0^*,\lambda^*)$. Moreover $\Phi_{\lambda}''(u_{\lambda})(u_{\lambda},u_{\lambda})>0$ for $\lambda\in(0,\lambda^*)$ and $\Phi_{\lambda^*}''(u_{\lambda^*})(u_{\lambda^*},u_{\lambda^*})=0$.
		\item[2)] For each $\lambda\in(0,\lambda_0^*+\varepsilon)$ problem \eqref{p} has a positive solution $w_\lambda$ which is a mountain pass critical point for $\Phi_\lambda$. 
		\item[3)] If $\lambda\in(0,\lambda_0^*)$ then $\Phi_\lambda(u_\lambda)<0$ while $\Phi_{\lambda_0^*}(u_{\lambda_0^*})=0$ and if $\lambda\in (\lambda_0^*,\lambda^*]$ then $\Phi_\lambda(u_\lambda)>0$.
		\item[4)] $\Phi_\lambda(w_\lambda)>0$ and $\Phi_\lambda(w_\lambda)>\Phi_\lambda(u_\lambda)$ for each $\lambda\in(0,\lambda_0^*+\varepsilon)$.
		\item[5)] If $\lambda>\lambda^*$ then the only solution $u\in H_0^1(\Omega)$ to the problem (\ref{p}) is the zero function $u=0$.
	\end{enumerate}
\end{theorem}
Kirchhoff type equations have been extensively studied in the literature. It was proposed by Kirchhoff in \cite{Kir} as an model to study some physical problems related to elastic string vibrations and since then it has been studied by many author, see for example the works of Lions \cite{Lions}, Alves et al. \cite{clafra}, Wu et al. \cite{wuch}, Zhang and Perera \cite{perzha} and the references therein. Physically speaking if one wants to study string or membrane vibrations, one is led to the equation \eqref{p}, where $u$ represents the displacement of the membrane, $|u|^{p-2}u$ is an external force, $a$ and $\lambda$ are related to some intrinsic properties of the membrane. In particular, $\lambda$ is related to the Young modulus of the material and it measures its stiffness.

Our main interest here is to analyze equation \eqref{p} with respect to the parameter $\lambda$ (stiffness) and provide a description of the bifurcation diagram. To this end, we will use the fibering method of Pohozaev \cite{poh} to analyse how the Nehari set (see Nehari \cite{neh,neh1}) change with respect to the parameter $\lambda$ and then apply this analysis to study bifurcation properties of the problem (\ref{p}) (see Chen et al. \cite{wuch} and Zhang et al. \cite{zhangnieto}). In fact, the extremal parameter $\lambda^*$ (see Il'yasov \cite{ilyasENMM}) which appears in the Theorem \ref{T} can be characterized variationally by 
\begin{equation*}
\lambda^*=C_{a,\gamma}\sup\left\{\left(\displaystyle\frac{\left(\int |u|^\gamma\right)^{\frac{1}{\gamma}}}{\left(\int |\nabla u|^2\right)^{\frac{1}{2}}}\right)^{\frac{2\gamma}{\gamma-2}}: u\in H_0^1(\Omega)\setminus\{0\}\right\},
\end{equation*}
where $C_{a,\gamma}$ is some positive constant. One can easily see from the last expression that $\lambda^*=C_{a,\gamma}S_\gamma^{\frac{2\gamma}{2-\gamma}}$, where $S_\gamma$ is best Sobolev constant for the embedding $H_0^1(\Omega)\hookrightarrow L^\gamma(\Omega)$.

In this work the extremal parameter $\lambda^*$ has the important role that if $\lambda>\lambda^*$ then the Nehari set is empty while if $\lambda\in(0,\lambda^*)$ then the Nehari set is not empty. Another interesting paramenter is $\lambda_0^*<\lambda^*$ which is characterized by the property that if $\lambda\in (0,\lambda_0^*)$, then $\inf_{u\in H_0^1(\Omega)}\Phi_\lambda(u)<0$ while if $\lambda\ge \lambda_0^*$ the infimum is zero. When $\lambda\in (0,\lambda_0^*]$ one can easily provide a Mountain Pass Geometry and a global minimizer for the functional $\Phi_\lambda$. Although here we characterize $\lambda_0^*$ variationally, one can see that the parameter $a^*$ defined in Theorem 1.3 (ii) of Sun and Wu \cite{sunwu} serves to the same purpose as $\lambda_0^*$ and hence our result for $\lambda\in(0,\lambda_0^*)$ is not new, however, when $\lambda>\lambda_0^*$ we could not find this result in the literature and in this case we need to provide some finer estimates on the Nehari sets in order to solve some technical issues to obtain again a Mountain Pass Geometry and a local minimizer for the functional $\Phi_\lambda$. 

The hypothesis $\gamma\in(2,4)$ has the fundamental role that it forces the problem to be superlinear, subcritical and it allows the existence of fiber maps with two critical points. The existence of these kinds of fiber maps implies multiplicity of solutions (at least two solutions) and once for $\lambda>\lambda^*$ there is no solution at all, the parameter $\lambda^*$ is a bifurcation point where these solutions collapses. We refer the reader to the recently works of Siciliano and Silva \cite{gaeka}, Il'yasov and Silva \cite{YaKa}, Silva and Macedo \cite{KaAb}, where the extremal parameters of some indefinite nonlinear elliptic problems were analyzed.  

Concerning the asymptotic behavior of the solutions when $\lambda\downarrow 0$ we prove the following

\begin{theorem}\label{T2}There holds
	\begin{enumerate}
		\item[i)] $\Phi_\lambda(u_\lambda)\to -\infty$ and $\|u_\lambda\|\to\infty$ as $\lambda\downarrow 0$.
		\item[ii)] $w_\lambda \to w_0$ in $H_0^1(\Omega)$ where $w_0\in H_0^1(\Omega)$ is a mountain pass critical point associated to the equation $-a\Delta w=|w|^{p-2}w$.
	\end{enumerate}
\end{theorem}

This work is organized as follows: In Section \ref{S2} we provide some definitions and prove technical results which will be used in the next sections. In Section \ref{S3} we show the existence of local minimizers for the functional $\Phi_\lambda$. In Section \ref{S4} we prove the existence of a mountain pass critical point for the functional $\Phi_\lambda$. In Section \ref{S5} we prove Theorem \ref{T}. In Section \ref{S6} we prove Theorem \ref{T2}. In Section \ref{S7} we provide a picture detailing the bifurcation diagram with respect to the energy and make some conjectures and in the Appendix we prove some auxiliary results.

\section{Technical Results}
\label{S2}
We denote by $\|u\|$ the standard Sobolev norm on $H_0^1(\Omega)$ and $\|u\|_\gamma$ the $L^\gamma(\Omega)$ norm. It follows from (\ref{energyfunctional}) that
\begin{equation*}
\Phi_\lambda(u)=\frac{a}{2}\|u\|^2+\frac{\lambda}{4}\|u\|^4-\frac{1}{\gamma}\|u\|_\gamma^\gamma,\ \forall\ u\in H_0^1(\Omega).
\end{equation*}
For each $\lambda>0$ consider the Nehari set
\begin{equation*}
\mathcal{N}_\lambda=\{u\in H_0^1(\Omega)\setminus\{0\}:\ \Phi_\lambda'(u)u=0 \}.
\end{equation*} 
To study the Nehari set we will make use of the fiber maps: for each $\lambda>0$ and $u\in H_0^1(\Omega)\setminus\{0\}$ define $\psi_{\lambda,u}:(0,\infty)\to \mathbb{R}$ by 
\begin{equation*}
\psi_{\lambda,u}(t)=\Phi_\lambda(tu).
\end{equation*}
It follows that 
\begin{equation*}
\mathcal{N}_\lambda=\{u\in H_0^1(\Omega)\setminus\{0\}:\ \psi'_{\lambda,u}(1)=0 \}.
\end{equation*}
We divide the Nehari set into three disjoint sets as follows:
\begin{equation*}
\mathcal{N}_\lambda=\mathcal{N}_\lambda^+\cup \mathcal{N}_\lambda^0\cup \mathcal{N}_\lambda^-,
\end{equation*}
where 
\begin{equation*}
\mathcal{N}_\lambda^+=\{u\in H_0^1(\Omega)\setminus\{0\}:\ \psi'_{\lambda,u}(1)=0,\ \psi''_{\lambda,u}(1)>0 \},
\end{equation*}
\begin{equation*}
\mathcal{N}_\lambda^0=\{u\in H_0^1(\Omega)\setminus\{0\}:\ \psi'_{\lambda,u}(1)=0,\ \psi''_{\lambda,u}(1)=0 \},
\end{equation*}
and
\begin{equation*}
\mathcal{N}_\lambda^-=\{u\in H_0^1(\Omega)\setminus\{0\}:\ \psi'_{\lambda,u}(1)=0,\ \psi''_{\lambda,u}(1)<0 \}.
\end{equation*}
By using the Implicit Function Theorem one can prove the following
\begin{lem}\label{Neharimanifold} If $\mathcal{N}_\lambda^+,\mathcal{N}_\lambda^-$ are non empty then $\mathcal{N}_\lambda^+,\mathcal{N}_\lambda^-$ are $C^1$ manifolds of codimension $1$ in $H_0^1(\Omega)$. Moreover if $u\in\mathcal{N}_\lambda^+\cup\mathcal{N}_\lambda^-$ is a critical point of  $\left(\Phi_\lambda\right)_{|\mathcal{N}_\lambda^+\cup\mathcal{N}_\lambda^-}$, then $u$ is a critical point of $\Phi_\lambda$. 
\end{lem}
In order to understand the Nehari set $\mathcal{N}_\lambda$ we study the fiber maps $\psi_{\lambda,u}$. Simple Analysis arguments show that
\begin{prop}\label{fibering}For each $\lambda>0$ and $u\in H_0^1(\Omega)\setminus\{0\}$, there are only three possibilities for the graph of $ \psi_{\lambda,u}$
	\begin{enumerate}
		\item[I)] The function $\psi_{\lambda,u}$ has only two critical points, to wit, $0<t_\lambda^-(u)<t_\lambda^+(u)$. Moreover, $t_\lambda^-(u)$ is a local maximum with $\psi''_{\lambda,u}(t_\lambda^-(u))<0$ and $t_\lambda^+(u)$ is a local minimum with $\psi''_{\lambda,u}(t_\lambda^+(u))>0$;   
		\item[II)] The function  $\psi_{\lambda,u}$ has only one critical point when $t>0$ at the value $t_\lambda(u)$. Moreover, $\psi''_{\lambda,u}(t_\lambda(u))=0$ and  $\psi_{\lambda,u}$ is increasing;
		\item[III)] The function  $\psi_{\lambda,u}$ is increasing and has no critical points.
	\end{enumerate}
\end{prop}
It follows from Proposition \ref{fibering} that $\mathcal{N}_\lambda^+,\mathcal{N}_\lambda^-$ are non empty if and only if the item $I)$ is satisfied. Therefore, it remains to show whether $I)$ is satisfied or not. For this purpose we study for what values of $\lambda$ there holds $\mathcal{N}_\lambda^0\neq \emptyset$. Note that $tu\in \mathcal{N}_\lambda^0$ for $t>0$ and $u\in H_0^1(\Omega)\setminus\{0\}$ if and only if 
\begin{equation*}
\left\{
\begin{aligned}
\psi'_{\lambda,u}(t) &= 0, \\
\psi''_{\lambda,u}(t) &= 0,
\end{aligned}
\right.
\end{equation*}
or equivalently 
\begin{equation}\label{extremal}
\left\{
\begin{aligned}
a\|u\|^2+\lambda \|u\|^4t^{2}-\|u\|_\gamma^\gamma t^{\gamma-2} &= 0, \\
a\|u\|^2+3\lambda \|u\|^4t^{2}-(\gamma-1)\|u\|_\gamma^\gamma t^{\gamma-2} &= 0.
\end{aligned}
\right.
\end{equation}
We solve the system (\ref{extremal}) with respect to the variable $(t,\lambda)$ to obtain for each $u\in H_0^1(\Omega)\setminus \{0\}$ a unique pair $(t(u),\lambda(u))$ such that
\begin{equation}\label{tu}
t(u)=\left(\frac{2a}{4-\gamma}\frac{\|u\|^2}{\|u\|_\gamma^\gamma}\right)^{\frac{1}{\gamma-2}},
\end{equation}
\begin{equation}\label{extrefunction}
\lambda(u)=C_{a,\gamma}\left(\frac{\|u\|_\gamma}{\|u\|}\right)^{\frac{2\gamma}{\gamma-2}},
\end{equation}
where
\begin{equation*}
C_{a,\gamma}=a\left(\frac{\gamma-2}{4-\gamma}\right)\left(\frac{4-\gamma}{2a}\right)^{\frac{2}{\gamma-2}}.
\end{equation*} 
We define the extremal parameter (see Il'yasov \cite{ilyasENMM})
\begin{equation}\label{extremalpara}
\lambda^*=\sup_{u\in H_0^1(\Omega)\setminus \{0\}}\lambda(u).
\end{equation}
We also consider another parameter which is defined as a solution of the system 
\begin{equation*}
\left\{
\begin{aligned}
\psi_{\lambda,u}(t)  &= 0, \\
\psi'_{\lambda,u}(t) &= 0,
\end{aligned}
\right.
\end{equation*}
or equivalently 
\begin{equation}\label{zeroenergy}
\left\{
\begin{aligned}
\frac{a}{2}\|u\|^2+\frac{\lambda}{4} t^{2}\|u\|^4-\frac{1}{\gamma}t^{\gamma-2}\|u\|_\gamma^\gamma &= 0, \\
a\|u\|^2+\lambda t^2\|u\|^4-t^{\gamma-2}\|u\|_\gamma^\gamma &= 0.
\end{aligned}
\right.
\end{equation}
Similar to the system (\ref{extremal}) we can solve the system \eqref{zeroenergy} with respect to the variable $(t,\lambda)$ to find a unique pair $(t_0(u),\lambda_0(u))$. Moreover, one can easily see that 
\begin{equation*}
\lambda_0(u)=C_{0,a,\gamma}\lambda(u),\ \forall\ u\in H_0^1(\Omega)\setminus \{0\}, 
\end{equation*}
where
\begin{equation*}
C_{0,a,\gamma}=2\left(\frac{2}{\gamma}\right)^{\frac{2}{\gamma-2}}.
\end{equation*} 
Observe that $C_{0,a,\gamma}<1$. We define
\begin{equation}\label{zeroenergypara}
\lambda_0^*=\sup_{u\in H_0^1(\Omega)\setminus \{0\}}\lambda_0(u).
\end{equation}
The functions $\lambda(u)$ and $\lambda_0(u)$ has the following geometrical interpretation 
\begin{prop}\label{fiberingvariation} For each $u\in H_0^1(\Omega)\setminus\{0\}$ there holds
	\begin{enumerate}
		\item[i)] $\lambda(u)$ is the unique parameter $\lambda>0$ for which the fiber map $\psi_{\lambda,u}$ has a critical point with second derivative zero at $t(u)$. Moreover, if $0<\lambda<\lambda(u)$, then $\psi_{\lambda,u}$ satisfies I) of Proposition \ref{fibering} while if $\lambda>\lambda(u)$, then $\psi_{\lambda,u}$ satisfies III) of Proposition \ref{fibering}.
		\item[ii)] $\lambda_0(u)$ is the unique parameter $\lambda>0$ for which the fiber map $\psi_{\lambda,u}$ has a critical point with zero energy at $t_0(u)$. Moreover, if $0<\lambda<\lambda_0(u)$, then $\inf_{t>0}\psi_{\lambda,u}(t)<0$ while if $\lambda>\lambda(u)$, then $\inf_{t>0}\psi_{\lambda,u}(t)=0$.
	\end{enumerate}
\end{prop}
\begin{proof} $i)$ The uniqueness of $\lambda(u)$ comes from equation \eqref{extremal}. Assume that $\lambda\in(0,\lambda(u))$, then $\psi_{\lambda,u}$ must satisfies $I)$ or $III)$ of Proposition \ref{fibering}. We claim that it must satisfies $I)$. Indeed, suppose on the contrary that it satisfies $III)$. Once
	\begin{equation*}
	\psi'_{\lambda(u),u}(t)>\psi'_{\lambda,u}(t)>0, \forall\ t>0,
	\end{equation*}
	we reach a contradiction since $\psi'_{\lambda(u),u}(t(u))=0$ where $t(u)$ is given by \eqref{tu}, therefore $\psi_{\lambda,u}$ must satisfies $I)$. Now suppose that $\lambda>\lambda(u)$, then 
	\begin{equation*}
	\psi'_{\lambda,u}(t)>\psi'_{\lambda(u),u}(t)\ge0, \forall\ t>0,
	\end{equation*}
	and hence $\psi_{\lambda,u}$ must satisfies $III)$. 
	
	$ii)$ The uniqueness of $\lambda_0(u)$ comes from equation \eqref{zeroenergy}. If $0<\lambda<\lambda_0(u)$ then from the definition we have
	\begin{equation*}
	\psi_{\lambda,u}(t_0(u))<\psi_{\lambda_0(u),u}(t_0(u))=0,
	\end{equation*}
	which implies that $\inf_{t>0}\psi_{\lambda,u}(t)<0$. If $\lambda>\lambda_0(u)$ then
	\begin{equation*}
	\psi_{\lambda,u}(t))>\psi_{\lambda_0(u),u}(t)\ge0, \forall t>0,
	\end{equation*}
	and therefore $\inf_{t>0}\psi_{\lambda,u}(t)=\psi_{\lambda,u}(0)=0$.
\end{proof}
Now we turn our attention to the parameters $\lambda^*$ and $\lambda_0^*$.
\begin{prop}\label{extremalparameter} There holds $\lambda_0^*<\lambda^*<\infty$. Moreover, there exists $ u\in H_0^1(\Omega)\setminus \{0\}$ such that $\lambda(u)=\lambda^*$ and $\lambda_0(u)=\lambda_0^*$.
\end{prop}
\begin{proof} Indeed, from the Sobolev embedding it follows that $\lambda_0,\lambda^*<\infty$. Now observe that $\lambda(u)$ is $0$-homogeneous, that is $\lambda(tu)=\lambda(u)$ for each $t>0$. It follows that there exists a sequence $u_n\in H_0^1(\Omega)\setminus \{0\}$ such that $\|u_n\|=1$ and $\lambda(u_n)\to \lambda^*$ as $n\to \infty$. We can assume that $u_n \rightharpoonup u$ in $H_0^1(\Omega)$ and $u_n\to u$ in $L^\gamma(\Omega)$. Moreover, from (\ref{extrefunction}) it follows that $u\neq 0$. We conclude that
	\begin{equation*}
	\lambda\left(\frac{u}{\|u\|}\right)=\lambda(u)\ge C_{a,\gamma}\left(\frac{\lim_{n\to \infty}\|u_n\|_\gamma}{\liminf_{n\to \infty}\|u_n\|}\right)^{\frac{2\gamma}{\gamma-2}}\ge \limsup_{n\to \infty}\lambda(u_n)=\lambda^*,
	\end{equation*}
	and hence $u_n\to u$ in $H_0^1(\Omega)$ and $u$ satisfies $\lambda(u)=\lambda^*$. Once $\lambda_0(u)$ is a mulitple of $\lambda(u)$ it follows also that $\lambda_0(u)=\lambda_0^*$ and from $C_{0,a,\gamma}<1$, we conclude that  $\lambda_0^*<\lambda^*$.
\end{proof}
As a consequence of Proposition \ref{extremalparameter} we have the following
\begin{prop}\label{Neharimanifoldsandzeroenergy} There holds
	\begin{enumerate}
		\item[i)] For each $\lambda\in (0,\lambda^*)$ we have that $\mathcal{N}_\lambda^+$ and $\mathcal{N}_\lambda^-$ are non empty. Moreover, if $\lambda>\lambda^*$ then $\mathcal{N}_\lambda=\emptyset$.
		\item[ii)] For each $\lambda\in (0,\lambda_0^*)$ there exists $u\in H_0^1(\Omega)\setminus\{0\}$ such that $\Phi_\lambda(u)<0$. Moreover, if $\lambda\ge\lambda_0^*$ then $\inf_{t>0}\psi_{\lambda,u}(t)=0$ for each $u\in H_0^1(\Omega)\setminus\{0\}$.
	\end{enumerate}
\end{prop}
\begin{proof} $i)$ From Proposition \ref{extremalparameter}, there exists $u\in H_0^1(\Omega)\setminus \{0\}$ such that $\lambda(u)=\lambda^*$. It follows from Proposition \ref{fiberingvariation} that for each $\lambda\in (0,\lambda^*)$ the fiber map $\psi_{\lambda,u}$ satisfies $I)$ of Proposition \ref{fibering} and hence $t_\lambda^-(u)u\in \mathcal{N}_\lambda^-$ and $t_\lambda^+(u)u\in \mathcal{N}_\lambda^+$. Now suppose that $\lambda>\lambda^*$, then it follows that $\lambda>\lambda^*\ge \lambda(u)$ for each $u\in H_0^1(\Omega)\setminus \{0\}$, which implies from Proposition \ref{fiberingvariation} that $\psi_{\lambda,u}$ satisfies $III)$ of Proposition \ref{fibering} and hence $\mathcal{N}_\lambda=\emptyset$.
	
	$ii)$ From Proposition \ref{extremalparameter}, there exists $u\in H_0^1(\Omega)\setminus \{0\}$ such that $\lambda(u)=\lambda_0^*$. It follows from Proposition \ref{fiberingvariation} that for each $\lambda\in (0,\lambda_0^*)$, there exists $t>0$ such that $\Phi_\lambda(tu)<0$. Now assume that $\lambda\ge \lambda_0^*$. Therefore $\lambda>\lambda_0^*\ge \lambda_0(u)$ for each $u\in H_0^1(\Omega)\setminus \{0\}$, which implies from Proposition \ref{fiberingvariation} that $\inf_{t>0}\psi_{\lambda,u}(t)=0$.
\end{proof}
From Proposition \ref{Neharimanifoldsandzeroenergy} we obtain the following nonexistence result.
\begin{cor}\label{nonexistence} For each $\lambda>\lambda^*$ the functional $\Phi_\lambda$ does not have critical points other than $u=0$.
\end{cor}
\begin{proof} Indeed, observe that for each $\lambda>\lambda^*$ there holds $\mathcal{N}_\lambda=\emptyset$.
\end{proof}
Now we turn our attention to some estimates which will prove to be useful on the next section. We start with:
\begin{prop}\label{distanceorigin} Suppose that $\lambda\in (0,\lambda^*]$, then there exists $r_\lambda>0$ such that $\|u\|\ge r_\lambda$ for each $u\in \mathcal{N}_\lambda$.
\end{prop}
\begin{proof} The existence of $r_\lambda$ is straightforward from
	\begin{equation*}
	a\|u\|^2+\lambda\|u\|^4-C\|u\|^\gamma\le a\|u\|^2+\lambda\|u\|^4-\|u\|_\gamma^\gamma=0,\ \forall\ u\in \mathcal{N}_\lambda,
	\end{equation*}
	where $C>0$ comes from the Sobolev embedding.	
\end{proof}
\begin{prop}\label{n0} For each $\lambda\in(0,\lambda^*]$, there holds
	\begin{equation*}
	\Phi_\lambda(u)=\frac{(\gamma-2)^2}{4\gamma(4-\gamma)}\frac{a^2}{\lambda},\ \forall\ u\in \mathcal{N}_\lambda^0.
	\end{equation*}
\end{prop}
\begin{proof} In fact, if $u\in \mathcal{N}_\lambda^0$, then
	\begin{equation}\label{E1}
	\left\{
	\begin{aligned}
	a\|u\|^2+\lambda \|u\|^4-\|u\|_\gamma^\gamma  &= 0, \\
	2a\|u\|^2+4\lambda \|u\|^4-\gamma\|u\|_\gamma^\gamma  &= 0.
	\end{aligned}
	\right.
	\end{equation}
	It follows from (\ref{E1}) that
	\begin{equation}\label{E2}
	\|u\|^2=\frac{\gamma-2}{4-\gamma}\frac{a}{\lambda}.
	\end{equation}	
	Moreover, from (\ref{E1}) we also have that 
	\begin{equation}\label{E3}
	\Phi_\lambda(u)=\frac{\gamma-2}{2\gamma}a\|u\|^2-\frac{4-\gamma}{4\gamma}\lambda\|u\|^4, \forall\ u\in \mathcal{N}_\lambda^0.
	\end{equation}
	We combine (\ref{E2}) with (\ref{E3}) to prove the proposition.
\end{proof}
We conclude this Section with some variational properties related to the functional $\Phi_\lambda$.
\begin{lem}\label{variatonal} For each $\lambda\in (0,\lambda^*)$ there holds
	\begin{enumerate}
		\item[i)] The functional $\Phi_\lambda$ is weakly lower semi-continuous and coercive.
		\item[ii)] Suppose that $u_n$ is a Palais-Smale sequence at the level $c\in\mathbb{R}$, that is $\Phi_\lambda(u_n)\to c$ and $\Phi_\lambda'(u_n)\to 0$ as $n\to \infty$, then $u_n$ converge strongly to some $u$.
		\item[iii)] There exist $C_\lambda>0$ and $\rho_\lambda>0$ satisfying
		\begin{equation*}
		\Phi_\lambda(u)\ge C_\lambda,\ \forall\ u\in H_0^1(\Omega),\ \|u\|=\rho_\lambda,
		\end{equation*}
		and
		\begin{equation*}
		\lim_{C_\lambda\to 0}\rho_\lambda=0.
		\end{equation*}
	\end{enumerate}
\end{lem}
\begin{proof} $i)$ is obvious. To prove $ii)$, observe from $i)$ that $u_n$ is bounded and therefore we can assume that $u_n\rightharpoonup u$ in $H_0^1(\Omega)$ and $u_n\to u$ in $L^\gamma(\Omega)$. From the limit  $\Phi_\lambda'(u_n)\to 0$ as $n\to \infty$ we infer that
	\begin{equation*}
	\limsup_{n\to \infty}[-(a+\lambda\|u_n\|^2)\Delta u_n(u_n-u)]=\limsup_{n\to \infty}|u_n|^{\gamma-2}u_n(u_n-u)=0,
	\end{equation*}
	which easily implies that $u_n \to u$ in $H_0^1(\Omega)$.
	
	$iii)$ It follows from the inequality
	\begin{equation*}
	\Phi_\lambda(u)\ge \frac{a}{2}\|u\|^2+\frac{\lambda}{4}\|u\|^4-\frac{C}{\gamma}\|u\|^\gamma,\ \forall\ H_0^1(\Omega),
	\end{equation*}
	where the constant $C$ is positive.
\end{proof}
\section{Local Minimizers for $\Phi_\lambda$}$\label{S3}$
In this section we prove the following 
\begin{prop}\label{GMl0} For each $\lambda\in(0,\lambda^*)$ the functional $\Phi_\lambda$ has a local minimizer $u_\lambda\in H_0^1(\Omega)\setminus\{0\}$. Moreover, if $\lambda\in(0,\lambda_0^*)$ then $\Phi_\lambda(u_\lambda)<0$ while $\Phi_{\lambda_0^*}(u_{\lambda_0^*})=0$ and if $\lambda\in(\lambda_0^*,\lambda^*)$ then $\Phi_\lambda(u_\lambda)>0$.
\end{prop}
\begin{rem} In fact if $\lambda\in(0,\lambda_0^*]$ then the local minimizer given by the Lemma \ref{GMBl0} is a global minimizer.
\end{rem}
We divide the proof of Proposition \ref{GMl0} in some Lemmas.
\begin{lem}\label{GMBl0} For each $\lambda\in(0,\lambda_0^*)$ the functional $\Phi_\lambda$ has a global minimizer $u_\lambda$ with negative energy.
\end{lem}
\begin{proof} It is a consequence of Lemma \ref{variatonal} and Proposition \ref{Neharimanifoldsandzeroenergy}.
\end{proof}
\begin{lem}\label{GMB00} The functional $\Phi_{\lambda_0^*}$ has a global minimizer $u_{\lambda_0^*}\neq 0$ with zero energy.
\end{lem}
\begin{proof} Suppose that $\lambda_n\uparrow\lambda_0^*$ as $n\to \infty$ and for each $n$ choose $u_n\equiv u_{\lambda_n}$, where $u_{\lambda_n}$ is given by Lemma \ref{GMBl0}. From the inequality $\Phi_{\lambda_n}(u_n)<0$ for each $n$ and Lemma \ref{variatonal} we obtain that $u_n$ is bounded. Therefore we can assume that $u_n\rightharpoonup u$ in $H_0^1(\Omega)$ and $u_n \to u$ in $L^\gamma(\Omega)$. From Lemma \ref{variatonal} we have that
	\begin{equation*}
	\Phi_{\lambda_0^*}(u)\leq\liminf_{n\to \infty}\Phi_{\lambda_n}(u_n)\le 0.
	\end{equation*}
	From Proposition \ref{Neharimanifoldsandzeroenergy} we conclude that $\Phi_{\lambda_0^*}(u)=0$ and hence $	\Phi_{\lambda_0^*}(u)=\lim_{n\to \infty}\Phi_{\lambda_n}(u_n)$. Therefore $u_n\to u$ in $H_0^1(\Omega)$ and from Proposition \ref{distanceorigin} we obtain that $u\neq 0$. If $u_{\lambda_0^*}\equiv u$ the proof is complete.
\end{proof}
\begin{rem}\label{Rm1} Observe that $\lambda_0^*(u_{\lambda_0^*})=\lambda_0^*$ and hence $\lambda^*(u_{\lambda_0^*})=\lambda^*$.
\end{rem}
In order to show the existence of local minimizers when $\lambda>\lambda_0^*$ we need the following definition: for $\lambda\in(0,\lambda^*)$ define
\begin{equation}\label{minneh} 
\hat{\Phi}_\lambda=\inf\{\Phi_\lambda(u):\ u\in\mathcal{N}_\lambda^+\cup\mathcal{N}_\lambda^0\}.
\end{equation}
\begin{rem} From the definitions, Proposition \ref{fibering} and Proposition \ref{Neharimanifoldsandzeroenergy} we conclude that 
	\begin{equation*}
	\hat{\Phi}_\lambda=\inf_{u\in H_0^1(\Omega)}\Phi_\lambda(u), \forall \lambda\in(0,\lambda_0^*].
	\end{equation*} 
\end{rem}
\begin{prop}\label{nearzeroenergy} For each $\lambda\in (\lambda_0^*,\lambda^*)$ there holds 
	\begin{equation*}
	\hat{\Phi}_\lambda<\frac{(\gamma-2)^2}{4\gamma(4-\gamma)}\frac{a^2}{\lambda}.
	\end{equation*}
\end{prop}
\begin{proof} Indeed, first observe from Remark \ref{Rm1} that $t_\lambda^+(u_{\lambda_0^*})$ is defined for each $\lambda\in (\lambda_0^*,\lambda^*)$. From Proposition \ref{decre} in the Appendix we know that $t_\lambda^-(u_{\lambda_0^*})<t_{\lambda_0^*}(u_{\lambda_0^*})<t_\lambda^+(u_{\lambda_0^*})$ for each $\lambda\in (\lambda_0^*,\lambda^*)$ and therefore
	\begin{eqnarray}\label{ttt}
	\hat{\Phi}_\lambda&\le& \Phi_\lambda(t_\lambda^+(u_{\lambda_0^*})u_{\lambda_0^*}) \nonumber\\       
	&<& \Phi_\lambda(t_{\lambda^*}(u_{\lambda_0^*})u_{\lambda_0^*}) \nonumber\\
	&<& \Phi_{\lambda^*}(t_{\lambda^*}(u_{\lambda_0^*})u_{\lambda_0^*}) \nonumber\\
	&=& \frac{(\gamma-2)^2}{4\gamma(4-\gamma)}\frac{a^2}{\lambda^*},\ \forall\ \lambda\in(\lambda_0^*,\lambda^*),
	\end{eqnarray}
	where the equality comes from Proposition \ref{n0}. We combine \eqref{ttt} with $\lambda<\lambda^*$ to complete the proof.
\end{proof}
\begin{lem}\label{ekeland} For each $\lambda\in(\lambda_0^*,\lambda^*)$ there exists $u_\lambda\in \mathcal{N}_\lambda^+$ such that $\Phi_\lambda(u_\lambda)=\hat{\Phi}_\lambda$.
\end{lem}
\begin{proof} Indeed, suppose that $u_n\in \mathcal{N}_\lambda^+\cup \mathcal{N}_\lambda^0$ satisfies $\Phi_{\lambda}(u_n)\to \hat{\Phi}_\lambda$. From Lemma \ref{variatonal} we have that $u_n$ is bounded and therefore we can assume that $u_n \rightharpoonup u$ in $H_0^1(\Omega)$ and $u_n\to u$ in $L^\gamma(\Omega)$. From $a\|u_n\|^2+\lambda\|u_n\|^{4}-\|u_n\|_\gamma^\gamma=0$ for all $n$ and Proposition \ref{distanceorigin} we conclude that $u\neq 0$. We claim that $u_n\to u$ in $H_0^1(\Omega)$. In fact, suppose on the contrary that this is false. It follows that 
	\begin{equation*}
	\psi'_{\lambda,u}(1)=a\|u\|^2+\lambda\|u\|^4-\|u\|_\gamma^\gamma< \liminf_{n\to \infty}(a\|u_n\|^2+\lambda\|u_n\|^{4}-\|u_n\|_\gamma^\gamma)=0,
	\end{equation*}
	and hence we conclude that the fiber map $\psi_{\lambda,u}$ satisfies $I)$ of Proposition \ref{fibering} and $t_\lambda^-(u)<1<t_\lambda^+(u)$. It follows that
	\begin{equation*}
	\Phi_\lambda(t_\lambda^+(u)u)< \Phi_\lambda(u)\le \liminf_{n\to \infty}\Phi_\lambda(u_n)=\hat{\Phi}_\lambda,
	\end{equation*}
	which is a contradiction since $t_\lambda^+(u)u\in \mathcal{N}_\lambda^+$. We conclude that $u_n\to u$ in $H_0^1(\Omega)$ and hence $\Phi_\lambda(u)=\hat{\Phi}_\lambda$. From Propositions \ref{n0} and \ref{nearzeroenergy} we obtain that $u\in \mathcal{N}_\lambda^+$.	
\end{proof}
\begin{proof}[Proof of Proposition \ref{GMl0}] The Lemmas \ref{GMBl0} and \ref{GMB00} guarantee the existence of a global minimizer $u_\lambda$ for the functional $\Phi_\lambda$ satisfying: if $\lambda\in(0,\lambda_0^*)$ then $\Phi_\lambda(u_\lambda)<0$ while $\Phi_{\lambda_0^*}(u_{\lambda_0^*})=0$. For $\lambda\in(\lambda_0^*,\lambda^*)$ we use Lemma \ref{ekeland} in order to obtain a local minimizer for the functional $\Phi_\lambda$. It remains to show that $\Phi_\lambda(u_\lambda)>0$ for $\lambda\in(\lambda_0^*,\lambda^*)$, however, once $\hat{\Phi}_{\lambda_0^*}=0$ this is a consequence of Proposition \ref{contidecre}.
\end{proof}

\section{Mountain Pass Solution for $\Phi_\lambda$}$\label{S4}$
In this Section we show the exsitence of a mountain pass type solution to equation \eqref{p}. In order to formulate our result we need to introduce some notation.
For each $\lambda\in(0,\lambda^*)$ define
\begin{equation}\label{MPE}
c_\lambda=\inf_{\varphi\in \Gamma_\lambda}\max_{t\in [0,1]}\Phi_\lambda(\varphi(t)),
\end{equation}
where $\Gamma_\lambda=\{\varphi\in C([0,1],H_0^1(\Omega)): \varphi(0)=0,\ \varphi(1)=\bar{u}_\lambda\}$ with $\bar{u}_\lambda=u_{\lambda_0^*}$ if $\lambda\in (0,\lambda_0^*]$ and $\bar{u}_\lambda=u_\lambda$ for $\lambda\in (\lambda_0^*,\lambda^*)$.
\begin{prop}\label{MPS}  There exists $\varepsilon>0$ such that for each $\lambda\in(0,\lambda_0^*+\varepsilon)$ one can find $w_\lambda\in H_0^1(\Omega)$ satisfying $\Phi_\lambda(w_\lambda)=c_\lambda$ and $\Phi'_\lambda(w_\lambda)=0$. Moreover $c_\lambda>0$ and $c_\lambda>\hat{\Phi}_\lambda$.
\end{prop}
To prove Proposotion \ref{MPS} we need some auxiliary results.
\begin{lem}\label{nearzeroenergyE} Given $\delta>0$, there exists $\varepsilon_\delta>0$ such that 
	\begin{equation*}
	0< \hat{\Phi}_\lambda\le \delta,\ \forall\lambda\in(\lambda_0^*,\lambda_0^*+\varepsilon_\delta).
	\end{equation*} 
\end{lem}
\begin{proof} The inequality $\hat{\Phi}_\lambda> 0$ follows from Proposition \ref{GMl0}. Let $u_{\lambda_0^*}$ be given as in Proposition \ref{GMB00}. Observe that if $\lambda\downarrow \lambda_0^*$, then $\Phi_\lambda(u_{\lambda_0^*})\to \Phi_{\lambda_0^*}(u_{\lambda_0^*})=0$. Moreover, since from Remark \ref{Rm1} the fiber map $\psi_{\lambda_0^*,u_{\lambda_0^*}}$ satisfies $I)$ of Proposition \ref{fibering}, we have from Proposition \ref{fiberingvariation} that $\lambda_0^*<\lambda(u_{\lambda_0^*})$. It follows that there exists $\varepsilon_1>0$ such that $\lambda_0^*+\varepsilon_1<\lambda(u_{\lambda_0^*})$. From Propositions \ref{fibering} and \ref{fiberingvariation}, for each $\lambda\in (\lambda_0^*,\lambda_0^*+\varepsilon_1)$, there exists $t_\lambda^+(u_{\lambda_0^*})>0$  such that $t_\lambda^+(u_{\lambda_0^*})u_{\lambda_0^*}\in \mathcal{N}_\lambda^+$. Note that $t_\lambda^+(u_{\lambda_0^*})\to 1$ as $\lambda\downarrow \lambda_0^*$ and therefore
	\begin{equation*}
	\hat{\Phi}_\lambda\le \Phi_\lambda(t_\lambda^+(u_{\lambda_0^*})u_{\lambda_0^*})\to \Phi_{\lambda_0^*}(u_{\lambda_0^*})=0,\ \lambda\downarrow \lambda_0^*.
	\end{equation*}
	If $\varepsilon_{2,\delta}>0$ is choosen in such a way that $\Phi_\lambda(t_\lambda^+(u_{\lambda_0^*})u_{\lambda_0^*})<\delta$ for each $\lambda\in (\lambda_0^*,\lambda_0^*+\varepsilon_{2,\delta})$, then we set $\varepsilon_\delta=\min\{\varepsilon_1,\varepsilon_{2,\delta}\}$ and the proof is complete.
\end{proof}
\begin{definition}\label{DEFI} For $\lambda\in(0,\lambda^*)$ denote
	\begin{equation}\label{D1}
	M_\lambda=\min\left\{C_\lambda,\frac{(\gamma-2)^2}{4\gamma(4-\gamma)}\frac{a^2}{\lambda}\right\},
	\end{equation}
	where  $C_\lambda$ is given by Lemma \ref{variatonal} and $\frac{(\gamma-2)^2}{4\gamma(4-\gamma)}\frac{a^2}{\lambda}$ is given by Proposition \ref{n0}. We assume that $\rho_\lambda<r_\lambda$ where both numbers are given by Lemma \ref{variatonal} and Proposition \ref{distanceorigin} respectively. Choose $0<\delta<M_\lambda$ and from Proposition \ref{nearzeroenergy} we take the corresponding $\varepsilon_\delta$.
\end{definition}
Now we are in position to prove Proposition \ref{MPS}
\begin{proof}[Proof of Proposition \ref{MPS}] The proof will be done once we show that the functional $\Phi_\lambda$ has a mountain pass geometry (remember that $u_\lambda$ is a local minimizer for $\Phi_\lambda$), however, one can see from Definition \ref{DEFI} that
	\begin{equation}\label{MPGI}
	\inf_{\|u\|=\rho_\lambda}\Phi_\lambda(u)\ge M_\lambda>\max\{\Phi_\lambda(0),\Phi_\lambda(\bar{u}_\lambda)\},
	\end{equation}
	which is the desired mountain pass geometry. It follows that $c_\lambda\ge M_\lambda>\Phi_\lambda(\bar{u}_\lambda)$ and $\Phi_\lambda(\bar{u}_\lambda)\ge \hat{\Phi}_\lambda$ if $\lambda\in (0,\lambda_0^*]$ and $\Phi_\lambda(\bar{u}_\lambda)= \hat{\Phi}_\lambda$ otherwise.
	
	We infer that there exists a Palais-Smale sequence for the functional $\Phi_\lambda$ at the level $c_\lambda$, that is, there exists $w_n\in H_0^1(\Omega)$ such that $\Phi_\lambda(w_n)\to c_\lambda$ and $\Phi_\lambda(w_n)\to 0$. From Lemma \ref{variatonal} we have that $w_n\to w$ in $H_0^1(\Omega)$ and hence $\Phi_\lambda(w_\lambda)=c_\lambda$ and $\Phi'_\lambda(w_\lambda)=0$.
\end{proof}
\section{Proof of Theorem \ref{T}}$\label{S5}$
In this Section we prove our main result
\begin{proof}[Proof of the Theorem \ref{T}] The existence of $u_\lambda$ and $w_\lambda$ are given by Propositions \ref{GMl0} and \ref{MPS}. Observe that $u_\lambda$ being a global minimizer for $\Phi_\lambda$ when $\lambda\in(0,\lambda_0^*]$ it is obviously a critical point for $\Phi_\lambda$ and hence a solution to (\ref{p}). If $\lambda\in(\lambda_0^*,\lambda^*)$ we saw in Lemma \ref{ekeland} that $u_\lambda\in \mathcal{N}_\lambda^+$ and hence from Lemma \ref{Neharimanifold} it is a critical point for the functional $\Phi_\lambda$. The case $\lambda=\lambda^*$ goes as following. Choose a sequence $\lambda\uparrow \lambda^*$ and a corresponding sequence $u_n\equiv u_{\lambda_n}$ such that $\Phi_{\lambda_n}(w_n)
	=\hat{\Phi}_{\lambda_n}$ and $\Phi'_{\lambda_n}(u_n)=0$ for each $n\in \mathbb{N}$.	Observe from the proof of Proposition \ref{nearzeroenergy} that 
	\begin{equation*}
	\hat{\Phi}_{\lambda_n}<  \frac{(\gamma-2)^2}{4\gamma(4-\gamma)}\frac{a^2}{\lambda^*},\ \forall\ n\in\mathbb{N},
	\end{equation*}
	and therefore from Lemma \ref{variatonal} we conclude that $u_n\to u$ in $H_0^1(\Omega)$. From Proposition \ref{contidecre} we obtain that
	\begin{equation*}
	\Phi_{\lambda^*}(u)=\lim_{n\to \infty}\Phi_{\lambda_n}(w_n)=\lim_{n\to \infty}\hat{\Phi}_{\lambda_n}>0,
	\end{equation*}
	and hence $u\neq 0$. By passing the limit it follows that $\Phi'_{\lambda^*}(u)=0$. Moreover from the definition of $\lambda^*$ we also obtain that $\Phi''_{\lambda^*}(u)(u,u)$=0. If we set $u_{\lambda^*}\equiv u$ the proof of Theorem \ref{T} items $1)$, $2)$ and $3)$ is complete.
	
	The item $4)$ is a consequence of Proposition \ref{MPS}. Item $5)$ is proved by using the fact that every critical point of $\Phi_\lambda$ lies in $\mathcal{N}_\lambda$ and Proposition \ref{Neharimanifoldsandzeroenergy}. To conclude we observe that standard arguments using the fact that $\Phi_\lambda(u)=\Phi_\lambda(|u|)$ provide positive solutions.
\end{proof}
\section{Asymptotic Behavior of $u_\lambda$ and $w_\lambda$ as $\lambda\downarrow 0$}$\label{S6}$
Define $\Phi_0:H_0^1(\Omega)\to \mathbb{R}$ by 
\begin{equation*}\label{lambda=0}
\Phi_0(u)=\frac{a}{2}\|u\|^2-\frac{1}{\gamma}\|u\|_\gamma^\gamma,
\end{equation*}
and observe that $\Phi_0(u_{\lambda_0^*})<\Phi_{\lambda_0^*}(u_{\lambda_0^*})=0$, where $u_{\lambda_0^*}$ is given by Theorem \ref{T}.  Define 
\begin{equation*}\label{MP00}
c_0=\inf_{\varphi\in \Gamma}\max_{t\in [0,1]}\Phi_0(\varphi(t)),
\end{equation*}
where $\Gamma=\{\varphi\in C([0,1]:H_0^1(\Omega)):\varphi(0)=0,\ \varphi(1)=u_{\lambda_0^*}  \}$.  Standard arguments provide a function $w_0\in H_0^1(\Omega)$ such that $\Phi_0(w_0)=c_0>0$ and $\Phi'_0(w_0)=0$. For $\lambda\in (0,\lambda^*_0)$, let us assume that $u_\lambda,w_\lambda$ are given by Theorem \ref{T}. In this section we prove the following
\begin{prop}\label{asym} There holds
	\begin{enumerate}
		\item[i)] $\Phi_\lambda(u_\lambda)\to -\infty$ and $\|u_\lambda\|\to\infty$ as $\lambda\downarrow 0$.
		\item[ii)] $w_\lambda \to w_0$ in $H_0^1(\Omega)$ where $w_0\in H_0^1(\Omega)$ satisfies $\Phi_0(w_0)=c_0$ and $\Phi'_0(w_0)=0$.
	\end{enumerate}
\end{prop}
\begin{proof} $i)$ Indeed, choose any $u\in H_0^1(\Omega)$ and suppose without loss of generality that $\lambda\in(0,\lambda(u))$. It follows from Proposition \ref{fibering} that $\psi_{\lambda,u}(t)\ge \psi_{\lambda,u}(t_\lambda^+(u))\ge \inf_{u\in H_0^1(\Omega)}\Phi_\lambda(u)=\hat{\Phi}_\lambda$. Now observe that for fixed $t>0$ there holds 
	\begin{equation}\label{AA1}
	\psi_{\lambda,u}(t)\to \frac{a}{2}\|u\|^2 t^2-\frac{1}{\gamma}\|u\|_\gamma^\gamma t^\gamma,\ \mbox{as}\ \lambda\downarrow 0.
	\end{equation} 
	Once 
	\begin{equation*}
	\lim_{t\to \infty}\left(\frac{a}{2}\|u\|^2 t^2-\frac{1}{\gamma}\|u\|_\gamma^\gamma t^\gamma\right)=-\infty,
	\end{equation*}
	it follows from \eqref{AA1} that given $M<0$ there exists $t>0$ and $\delta>0$ such that if $\lambda\in(0,\delta)$, then $\psi_{\lambda,u}(t)<M$ and hence $\hat{\Phi}_\lambda<M$, which proves that $\Phi_\lambda(u_\lambda)\to -\infty$ as $\lambda\downarrow 0$. One can easily infer from the last convergence that $\|u_\lambda\|\to\infty$ as $\lambda\downarrow 0$.
\end{proof}
To prove the item $ii)$ of Proposition \ref{asym} we need to establish some results. 
\begin{lem}\label{las1}
	The function $[0,\lambda_0^*)\ni\lambda\mapsto c_\lambda=\Phi_\lambda(w_\lambda)$ is non-decreasing. Moreover $c_\lambda\to c_0$ as $\lambda\downarrow 0$.
\end{lem}
\begin{proof} First observe that $\Gamma_\lambda=\Gamma$ for each $\lambda\in (0,\lambda_0^*]$. Suppose that $0\le\lambda<\lambda'<\lambda_0^*$ and fix any $\varphi\in \Gamma$. It follows that $\max_{t\in [0,1]}\Phi_\lambda(\varphi(t))<\max_{t\in [0,1]}\Phi_{\lambda'}(\varphi(t))$ and by taking the infimum in both sides we conclude that $c_\lambda\le c_{\lambda'}$.
	
	Once $c_\lambda$ is non-decreasing, we can assume that $c_\lambda\to c\ge c_0$ as $\lambda \downarrow 0$. Suppose on the contrary that $c>c_0$. Given $\delta>0$ such that $c_0+\varepsilon<c$ choose $\varphi\in \Gamma$ such that $c_0\le \max_{t\in [0,1]}\Phi_0(\varphi(t))<c_0+\delta$. If $\lambda$ is sufficiently close to $0$, then  $c_0\le \max_{t\in [0,1]}\Phi_0(\varphi(t))< \max_{t\in [0,1]}\Phi_\lambda(\varphi(t))<c_0+\delta$ and consequently $c_0\le c_\lambda<c_0+\delta<c<c_\lambda$ which is clearly a contradiction and therefore $c_\lambda\to c_0$ as $\lambda\downarrow 0$.
\end{proof}
Now we may finish the proof of Proposition \ref{asym}:
\begin{proof}[Proof of $ii)$ of Proposition \ref{asym}] Indeed, suppose that $\lambda_n\downarrow 0$ and for each $n\in \mathbb{N}$ choose $w_n\equiv w_{\lambda_n}$ such that $\Phi_{\lambda_n}(w_n)=c_{\lambda_n}$ and $\Phi'_{\lambda_n}(w_n)=0$. We claim that $\lambda_n\|w_n\|^4\to 0$ as $n\to \infty$. In fact, for each $n$ we can find a path $\varphi_n\in \Gamma_{\lambda_n}=\Gamma$ and a function $v_n$ such that $\Phi_{\lambda_n}(v_n)=\max_{t\in[0,1]}\Phi_{\lambda_n}(\varphi(t))$ and
	\begin{equation}\label{tt1}
	0<\Phi_{\lambda_n}(v_n)-c_{\lambda_n}\to 0,\ \|v_n-w_n\|\to 0,\ \|v_n-w_n\|_\gamma\to 0,\ \mbox{as}\ n\to \infty.
	\end{equation} 
	Now observe from the definition of $c_0$, Lemma \ref{las1} and (\ref{tt1}) that 
	\begin{equation}\label{tt2}
	0<\lim_{n\to \infty}\Phi_0(v_n)-c_0\le \lim_{n\to \infty}\Phi_{\lambda_n}(v_n)-c_0=\lim_{n\to \infty}(\Phi_{\lambda_n}(v_n)-c_{\lambda_n})=0.
	\end{equation}
	It follows from \eqref{tt1} and \eqref{tt2} that
	\begin{equation*}
	\frac{a}{2}\|v_n\|^2-\frac{1}{p}\|v_n\|_\gamma^\gamma\to 0\ \mbox{and}\ \frac{a}{2}\|v_n\|^2+\frac{\lambda_n}{4}\|v_n\|^4-\frac{1}{p}\|v_n\|_\gamma^\gamma\to 0,\ \mbox{as}\ n\to \infty,
	\end{equation*}
	which implies that $\lambda_n\|v_n\|^4\to 0$ as $n\to \infty$. From \eqref{tt1} we conclude that
	\begin{equation*}
	|\lambda_n\|w_n\|^4-\lambda_n\|v_n\|^4|\to 0,\ \mbox{as}\ n\to \infty,
	\end{equation*}
	and hence $\lambda_n\|w_n\|^4\to 0$ as $n\to\infty$ as we desired. Now note from the equations $\Phi_{\lambda_n}(w_n)=c_{\lambda_n}$ and $\Phi'_{\lambda_n}(w_n)=0$, $n\in\mathbb{N}$ that 
	\begin{equation}\label{pt1}
	\left\{
	\begin{aligned}
	\frac{a}{2}\|w_n\|^2+\frac{\lambda_n}{4} \|w_n\|^4-\frac{1}{\gamma}\|w_n\|_\gamma^\gamma  &= c_{\lambda_n}, \\
	a\|w_n\|^2+\lambda_n \|w_n\|^4-\|w_n\|_\gamma^\gamma  &= 0, 
	\end{aligned}
	\right.
	\end{equation}
	which combined with the limit $\lambda_n\|w_n\|^4\to 0$ as $n\to\infty$ and the Lemma \ref{las1} implies that 
	\begin{equation*}
	\left\{
	\begin{aligned}
	\frac{a}{2}\lambda_n\|w_n\|^2-\frac{\lambda_n}{\gamma}\|w_n\|_\gamma^\gamma  &= o(1), \\
	a\lambda_n\|w_n\|^2-\lambda_n\|w_n\|_\gamma^\gamma  &= 0.
	\end{aligned}
	\right.
	\end{equation*}
	We multiply the first equation by $-\gamma$ and sum with the second equation to obtain that 
	\begin{equation*}
	\left(-\frac{\gamma}{2}+1\right)a\lambda_n\|w_n\|^2=o(1),
	\end{equation*}
	which implies that $\lambda_n\|w_n\|^2\to 0$ as $n\to \infty$. Now we claim that $\|w_n\|$ is bounded. In fact, suppose on the contrary that up to a subsequence $\|w_n\|\to \infty$ as $n\to \infty$. From \eqref{pt1} we obtain that 
	\begin{equation*}
	\left\{
	\begin{aligned}
	\frac{a}{2}+\frac{\lambda_n}{4} \|w_n\|^2-\frac{1}{\gamma}\frac{\|w_n\|_\gamma^\gamma}{\|w_n\|^2}  &= o(1), \\
	a+\lambda_n \|w_n\|^2-\frac{\|w_n\|_\gamma^\gamma}{\|w_n\|^2}  &= 0.
	\end{aligned}
	\right.
	\end{equation*}
	Once $\lambda_n\|w_n\|^2\to 0$ as $n\to \infty$ we conclude that $\gamma=2$ which is a contradiction. Since $\|w_n\|$ is bounded we obtain that $\Phi_0(w_n)\to c_0$ and $\Phi'_0(w_n)\to 0$ as $n\to \infty$ and hence $w_n\to w_0$ as $n\to \infty$, where $w_0$ satisfies $\Phi_0(w_0)=c_0$ and $\Phi'_0(w_0)=0$.
\end{proof}
\begin{proof}[Proof of the Theorem \ref{T2}] It is a consequence of Proposition \ref{asym}.
\end{proof}
\section{Some Conclusions and Remarks}$\label{S7}$
If we plot the energy of the two solutions as a function of $\lambda$ we obtain the following picture:

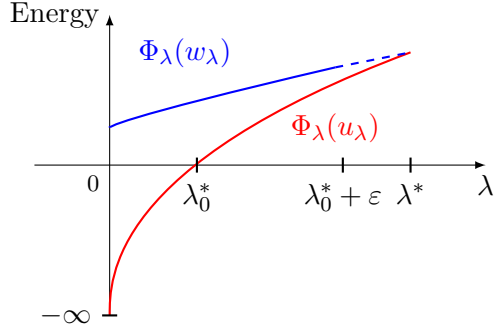
\begin{figure}[h]
	\centering
	\begin{tikzpicture}[>=latex]
	\draw[->] (-1,0) -- (5,0) node[below] {$\lambda$};
	\foreach \x in {}
	\draw[shift={(\x,0)}] (0pt,2pt) -- (0pt,-2pt) node[below] {\footnotesize $\x$};
	\draw[->] (0,-2) -- (0,2) node[left] {$\mbox{Energy}$};
	\foreach \y in {}
	\draw[shift={(0,\y)}] (2pt,0pt) -- (-2pt,0pt) node[left] {\footnotesize $\y$};
	\node[below left] at (0,0) {\footnotesize $0$};
	\draw[red,thick] (0,-2) .. controls (0,-1.5) and (0,0) .. (4,1.5);
	\draw[blue,thick] (0,.5) .. controls (0,0.5) and (0,0.6) .. (3,1.3);
	\draw[blue,thick,dashed] (3,1.3) .. controls (3,1.3) and (4,1.5) .. (4,1.5);
	\draw [thick] (1.16,-.1) node[below]{$\lambda_0^*$} -- (1.16,0.1); 
	\draw [thick] (3.1,-.1) node[below]{$\lambda_0^*+\varepsilon$} -- (3.1,0.1); 
	\draw [thick] (4,-.1) node[below]{$\lambda^*$} -- (4,0.1); 
	\draw [thick] (-.1,-2) node[left]{$-\infty$} -- (.1,-2); 
	\node[] at (1,1.5) { {\color{blue}$\Phi_\lambda(w_\lambda)$}};
	\node[] at (3,0.5) { {\color{red}$\Phi_\lambda(u_\lambda)$}};
	\end{tikzpicture}
	\caption{Energy depending on $\lambda$} \label{fig:M1}
	
\end{figure}
Observe from Proposition \ref{contidecre} that the energy of the local minimum depending on $\lambda$ is continuous and increasing (red plot) and although we could not prove it, we believe that the same holds true  for the energy of the mountain pass solution (blue plot). We also believe that $\lambda^*$ is a bifurcation turning point, that is, the two types of solutions must coincide at $\lambda^*$ as Figure 1 suggests.

\appendix
\section{}

\begin{prop}\label{contidecre} The function $(0,\lambda^*)\ni \lambda\mapsto\hat{\Phi}_\lambda$ is continuous and increasing.
\end{prop}
\begin{proof} First we prove that $(0,\lambda^*)\ni \lambda\mapsto\hat{\Phi}_\lambda$ is decreasing. Indeed, suppose that $\lambda<\lambda'$. From Lemmas \ref{GMBl0}, \ref{GMB00} and \ref{ekeland}, there exists $u_{\lambda'}$ such that $\hat{\Phi}_{\lambda'}=\Phi_{\lambda'}(u_{\lambda'})$. Since the fiber map $\psi_{\lambda',u_{\lambda'}}$ obvioulsy satisfies I) of Proposition \ref{fibering} it follows from Proposition \ref{fiberingvariation} that $\psi_{\lambda,u_{\lambda'}}$ also satisfies I) of Proposition \ref{fibering} and then	 
	\begin{equation*}
	\hat{\Phi}_{\lambda}\le \Phi_{\lambda}(t_{\lambda}^+(u_{\lambda'})u_{\lambda'})<\Phi_{\lambda}(t_{\lambda'}^+(u_{\lambda'})u_{\lambda'})=\Phi_{\lambda'}(u_{\lambda'})=\hat{\Phi}_{\lambda'}.
	\end{equation*}
	Now we prove that $(0,\lambda^*)\ni \lambda\mapsto\hat{\Phi}_\lambda$ is continuous. In fact, suppose that $\lambda_n\uparrow \lambda\in (0,\lambda^*)$ and choose $u_n\equiv u_{\lambda_n}$ such that $\hat{\Phi}_{\lambda_n}= \Phi_{\lambda_n}(u_n)$ for all $n$. Similar to the proof of Lemma \ref{ekeland} we may assume that $u_n\to u\in\mathcal{N}_\lambda^+$. We claim that $\hat{\Phi}_{\lambda_n}\to \hat{\Phi}_\lambda$ as $n\to \infty$. Indeed, once $(0,\lambda^*)\ni \lambda\mapsto\hat{\Phi}_\lambda$ is increasing, we can assume that $\hat{\Phi}_{\lambda_n}< \hat{\Phi}_\lambda$ for each $n$ and $\hat{\Phi}_{\lambda_n}\to \Phi_\lambda(u) \le\hat{\Phi}_\lambda$ as $n\to \infty$, wich implies that $\Phi_\lambda(u) =\hat{\Phi}_\lambda$.
	
	Now suppose that $\lambda_n\downarrow \lambda\in (0,\lambda^*)$. Once $(0,\lambda^*)\ni \lambda\mapsto\hat{\Phi}_\lambda$ is increasing, we can assume that $\hat{\Phi}_{\lambda_n}> \hat{\Phi}_\lambda$ for each $n$ and $\lim_{n\to \infty}\hat{\Phi}_{\lambda_n}\ge\hat{\Phi}_\lambda$. Choose $u_\lambda$ such that $\hat{\Phi}_\lambda=\Phi_\lambda(u_\lambda)$ and observe that $\hat{\Phi}_\lambda\le \lim_{n\to \infty}\hat{\Phi}_{\lambda_n}\le\lim_{n\to \infty}\Phi_{\lambda_n}(t_{\lambda_n}^+(u_\lambda)u_\lambda)=\hat{\Phi}_\lambda$.
	
\end{proof}
For the next proposition we assume that $u_{\lambda_0^*}$ is given as in Lemma \ref{GMB00} and $t(u_{\lambda_0^*})$ is defined in \eqref{tu}. Observe from Remark \ref{Rm1} that $t_\lambda^+(u_{\lambda_0^*})$ is well defined for each $\lambda\in (0,\lambda^*)$.
\begin{prop}\label{decre} There holds
	\begin{enumerate}
		\item[i)] The function $(0,\lambda^*)\ni \lambda\mapsto t_\lambda^+(u_{\lambda_0^*})$ is decreasing and continuous.
		\item[ii)] The function $(0,\lambda^*)\ni \lambda\mapsto t_\lambda^-(u_{\lambda_0^*})$ is increasing and continuous.
	\end{enumerate} Moreover 
	\begin{equation*}
	\lim_{\lambda\uparrow \lambda^*}t_\lambda^+(u_{\lambda_0^*})=\lim_{\lambda\uparrow \lambda^*}t_\lambda^-(u_{\lambda_0^*})=t(u_{\lambda_0^*}).
	\end{equation*}
\end{prop}
\begin{proof} Indeed, let $t_\lambda\equiv t_\lambda^+(u_{\lambda_0^*})$ and note that $t_\lambda$ satisfies $\psi'_\lambda(t_\lambda)=0$ for each $\lambda\in (0,\lambda^*)$. By implicit differentiation and the fact that $\psi''_\lambda(t_\lambda)>0$, we conclude that $(0,\lambda^*)\ni \lambda\mapsto t_\lambda^+(u_{\lambda_0^*})$ is decreasing and continuous, which proves $i)$ The proof of $ii)$ is similar and the limits
	\begin{equation*}
	\lim_{\lambda\uparrow \lambda^*}t_\lambda^+(u_{\lambda_0^*})=\lim_{\lambda\uparrow \lambda^*}t_\lambda^-(u_{\lambda_0^*})=t(u_{\lambda_0^*}),
	\end{equation*}
	are straightforward from the definitions.
	
\end{proof}

\newpage


\bibliographystyle{amsplain}
\bibliography{Ref}

\end{document}